\documentclass[11pt]{amsart}
\usepackage{amsthm,latexsym,amssymb,amsmath,mathrsfs,paralist}
\usepackage[all]{xy}

\newtheorem{theorem}{Theorem}[section]
\newtheorem{lemma}[theorem]{Lemma}

\newcommand{\Lie}{\mathrm{Lie}}

\newcommand{\Q}{\overline{\mathbb Q}}
\newcommand{\C}{{\mathbb C}}

\begin{document}

\title{$p$-adic non-commutative analytic subgroup theorem}
\author[D. H. Pham]{Duc Hiep Pham}

\begin{abstract}
In this paper, we formulate and prove the so-called $p$-adic non-commutative analytic subgroup theorem. This result is seen as the $p$-adic analogue of a recent theorem given by Yafaev in \cite{y}. 
\end{abstract}
\maketitle

\noindent 2020 Mathematical Subject Classification: 14L10 (22E35, 11F85, 11J81)\\
Keywords: algebraic groups, analytic subgroup theorem, $p$-adic, transcendence.  

\section{Introduction}\label{1}

It is known that the analytic subgroup theorem is considered as one of the most powerful theorems in complex transcendental number theory. The theorem was established by W\"{u}stholz in the 1980's based on a very deep auxiliary result on multiplicity estimates on group varieties (see \cite{w3} and \cite{w4}). To present the theorem, we start with $G$ a commutative algebraic group defined over $\Q$, that is a smooth quasi-projective variety defined over $\Q$ with a commutative group law for which the composition map $G\times G\rightarrow G,\;(g,h)\mapsto g\circ h$ and the inverse map $G\rightarrow G,\; g\mapsto g^{-1}$ are regular morphisms between algebraic varieties defined over $\Q$. Then the set $G(\C)$ of complex points of $G$ is a complex Lie group, and one has the exponential map $\exp_{G(\C)}:\Lie(G(\C))\rightarrow G(\C)$ of $G$. We say that an element $u\in\Lie(G(\C))$ is an algebraic point of the exponential map of $G$ if $\exp_{G(\C)}(u)\in G(\Q)$. The following theorem is a direct consequence of \cite[Theorem 1]{w1}, and it is still called {\it the analytic subgroup theorem} (see also \cite[Theorem 1.2]{pt}).
\begin{theorem}[W\"{u}stholz]\label{wu} Let $G$ be a connected commutative algebraic group defined over $\Q$. Let $u\in\Lie(G(\C))$ be an algebraic point of the exponential map of $G$ and $V_u$ the smallest $\Q$-vector subspace of $\Lie(G)$ such that $u$ lies in the complex vector space $V_u\otimes_{\Q}\C$. Then there exists a connected algebraic subgroup $H_u$ of $G$ defined over $\Q$ satisfying $\Lie(H_u)=V_u$.    
\end{theorem}
Many results in transcendence theory can be deduced from W\"{u}stholz analytic subgroup theorem as consequences by choosing suitable commutative algebraic groups. For instance, the theorem generalizes Baker's famous result on linear forms in logarithms and its elliptic analogue in full generality (see \cite[Chapter 6]{bw}). It also implies some important results concerning linear independence on abelian varieties between periods (see \cite[Chapter 1]{pt}). Furthermore, W\"{u}stholz himself has recently obtained a nice application on elliptic and abelian periods spaces (see\cite{w5}).

It is able to extend Theorem \ref{wu} to the non-commutative case. Note that in this case, there is still the exponential map between the Lie algebra and the Lie group as the commutative case. The following theorem has been recently proved by Yafaev (see \cite{y}). 
\begin{theorem}[Yafaev]\label{y}
Let $G$ be a connected algebraic group defined over $\Q$. Let $u\in\Lie(G(\C))$ be an algebraic point of the exponential map of $G$ and $V_u$ the smallest $\Q$-vector subspace of $\Lie(G)$ such that $u$ lies in the complex vector space $V_u\otimes_{\Q}\C$. Then there exists a connected algebraic subgroup $H_u$ of $G$ defined over $\Q$ which is commutative and satisfies $\Lie(H_u)=V_u$.
\end{theorem}
In 2015, Fuch and Pham obtained the $p$-adic analytic subgroup theorem (see \cite{fp1}), and the aim of this paper is to give a such
 $p$-adic analogue for Theorem \ref{y} above. In the $p$-adic setting, similar to \cite[Theorem 2.2]{fp1}, we also use the $p$-adic logarithm map which is still valid in the non-commutative case. Fix a prime number $p$, let $\Q_p$ be the field of $p$-adic numbers. Denote by $\C_p$ the completion of the algebraic closure of $\Q_p$ (with respect to the normalized $p$-adic absolute value). Let $G$ now be an algebraic group defined over $\Q$. Consider the set $G(\C_p)$ of $\C_p$-points of $G$. This is a Lie group over $\C_p$. Denote by $G(\C_p)_f$ the set of $x$ in $G(\C_p)$ satisfying there exists a strictly increasing sequence $(n_i)$ of positive integers for which $x^{n_i}$ tends to the identity element of $G(\C_p)$ as $i$ tends to infinity. It follows from \cite[Chapter III, 7.6]{bo} that there is a $\C_p$-analytic map $\log_{G(\C_p)}:G(\C_p)_f\rightarrow \Lie(G(\C_p))$, and it is called the {\it $p$-adic logarithm map} of $G$. Then {\it the $p$-adic non-commutative analytic subgroup theorem} is formulated as follows.
 
\begin{theorem}\label{mth}
Let $G$ be a connected algebraic group defined over $\overline{\mathbb Q}$ of positive dimension, and $\log_{G(\C_p)}:G(\C_p)_f\rightarrow \Lie(G(\C_p))$ the $p$-adic logarithm map of $G$. Let $\gamma\in G(\C_p)_f$ be an algebraic point of $G(\Q)$ with $\log_{G(\mathbb C_p)}(\gamma)\ne 0$ and $V_{\gamma}$ the smallest $\Q$-vector subspace of $\Lie(G)$ such that the $p$-adic vector space $V_{\gamma}\otimes_{\Q}\C_p$ containing the point $\log_{G(\mathbb C_p)}(\gamma)$. Then
there exists a connected commutative algebraic subgroup $H_{\gamma} \subseteq G$ defined over $\overline{\mathbb Q}$ of positive dimension satisfying $\gamma\in H_{\gamma}(\overline{\mathbb Q})$ and $\Lie(H_{\gamma})=V$. 
\end{theorem}
The proof of our theorem follows closely that of Yafaev's theorem. We also reduce the general case to the commutative case and then apply the $p$-adic analytic subgroup theorem. The main difference is that, in the $p$-adic setting, the $p$-adic exponential map is only defined locally on an open subgroup of the $p$-adic Lie algebra of the given algebraic group. However, by using the $p$-adic logarithm map and by taking a certain power of the algebraic point, we are able to still keep the important similar arguments as in the complex case, and based on this to get the derised result.

\section{The p-adic exponential map and related commutative algebraic subgroups}\label{2}
In this section, we first briefly recall some background on exponential and logarithm maps over $p$-adic fields. The main reference we follow here is \cite{bo}. For an algebraic group $G$ defined over $\Q$, by \cite[Chapter III, 7.2, 7.6]{bo} there exist an open subgroup $U_p$ of $\Lie(G(\C_p))$ and an analytic map $\exp_{G(\C_p)}:U_p\rightarrow G(\C_p)$ such that $\exp_{G(\C_p)}(U_p)$ is an open subgroup of $G(\C_p)$, and the map $\exp_{G(\C_p)}$ induces an isomprphism between $U_p$ and $\exp_{G(\C_p)}(U_p)\subset G(\C_p)_f$, whose inverse is the restriction of $\log_{G(\C_p)}$ to $\exp_{G(\C_p)}(G_p)$. The map $\exp_{G(\C_p)}$ is called the {\it $p$-adic exponential map of $G$}. For an element $u$ in $U_p$ and for a subset $S$ in $\C_p$, define $S\cdot u$ by the set $\{su; s\in S\}$. 
The following lemma plays a central role in the proof of the main theorem which allows us to construct a relevant commutative algebraic subgroup of $G$ defined over $\Q$.  

\begin{lemma}\label{lm}
Let $G$ be an algebraic group defined over $\Q$. Let $u$ be an element in $U_p$ with $\exp_{G(\C_p)}(u)\in G(\Q)$. Denote by $G_u$ the Zariski closure of the set 
$\exp_{G(\C_p)}((\C_p\cdot u)\cap U_p)$ in $G(\C_p)$. Then $G_u$ is a commutative algebraic subgroup of $G$ defined over $\Q$ and $\exp_{G(\C_p)}(u)\in G_u(\Q)$.
\end{lemma}
\begin{proof}
Put $\Phi_u:=\exp_{G(\C_p)}((\C_p\cdot u)\cap U_p)$. We first see that any two elements $x,y\in \C_p\cdot u$ commute, and therefore 
$$\exp_{G(\C_p)}(x)\exp_{G(\C_p)}(y)=\exp_{G(\C_p)}(x+y)=\exp_{G(\C_p)}(y+x)=\exp_{G(\C_p)}(y)\exp_{G(\C_p)}(x),$$
by Campbell-Hausdorff formula (see \cite[Section 16]{sch}). 
In particular, this shows that $\Phi_u$ is a commutative subgroup of $G(\C_p)$. It follows from \cite[Lemma 1.40]{m} that the Zariski closure $G_u$ of $\Phi_u$ is an algebraic subgroup of $G(\C_p)$.
We now consider the commutator morphism $f: G_u\times G_u\rightarrow G_u$ given by $f(x,y)=[x,y]$. Then $f$ is an algebraic morphism and trivial on $\Phi_u\times \Phi_u$. This gives $\Phi_u\times \Phi_u$ is contained in the algebraic set $f^{-1}(e)$, where $e$ denotes the identity element of $G$. But $\Phi_u\times \Phi_u$ is Zariski-dense in $G_u\times G_u$, this implies that the morphism $f$ must be trivial on $G_u\times G_u$. In other words, $G_u$ is commutative.
It remains to show $G_u$ is defined over $\Q$. If $u=0$ then $G_u$ is the trivial group, and hence clearly defined over $\Q$. Assume $u\ne 0$, let $\Gamma_u$ denote the image of the set $(\mathbb Q\cdot u)\cap U_p$ under the map $\exp_{G(\C_p)}$. As above, we also have $\Gamma_u$ is a subgroup of $G(\C_p)$, and therefore by \cite[Lemma 1.40]{m} again the Zariski closure $G'_u$ of $\Gamma_u$ in $G(\C_p)$ is an algebraic group. On the other hand, since $\exp_{G(\C_p)}(u)\in G(\Q)$, it follows that $\exp_{G(\C_p)}(u)\in G_u(\Q)$. Note that $\exp_{G(\C_p)}(mu)=\exp_{G(\C_p)}(u)^m$ for all $m\in \mathbb Z$, and this leads to $\Gamma_u$ is contained in $G_u(\Q)$. Hence, the algebraic group $G'_u$ must be defined over $\Q$. Moreover, the group $\Gamma_u$ is infinite (since $u\ne 0$), thus the algebraic group $G_u'$ has positive dimension. 

Next, one has $\exp_{G(\C_p)}(u)\in G'_u(\Q)\cap G(\C_p)_f\subseteq G'_u(\C_p)_f$. This gives 
$$u=\log_{G(\C_p)}\big(\exp_{G(\C_p)}(u)\big)=\log_{G'_u(\C_p)}\big(\exp_{G(\C_p)}(u)\big)\in \Lie(G'_u(\C_p)).$$
Since $\Lie(G'_u(\C_p))$ is a vector space over $\C_p$, it follows that the line $\C_p\cdot u$ is also contained in $\Lie(G'_u(\C_p))$. From this, we get 
$$\Phi_u\subseteq\exp_{G(\C_p)}\big(\Lie(G'_u(\C_p))\cap U_p\big)
= \exp_{G'_u(\C_p)}\big(\Lie(G'_u(\C_p))\cap U_p\big)\subseteq G'_u(\C_p).
$$
This allows us to deduce that $G_u(\C_p)\subset G'_u(\C_p)$. But, by definition, $G'_u(\C_p)$ is obviously contained in $G_u(\C_p)$. Therefore, we conclude that $G_u=G'_u$, and this completes the proof of the lemma.             

\section{Proof of the main theorem}
We are now ready to prove the main theorem. Let $\exp_{G(\C_p)}:U_p\rightarrow G(\C_p)$ be the $p$-adic exponential map of $G$ defined as in Section \ref{2}. By definition of the set $G(\C_p)_f$, there is a positive integer $k$ such that $\gamma^k\in U_p$. Let $u$ be the point $\log_{G(C_p)}(\gamma^k)$. We have 
$$\exp_{G(\C_p)}(u)=\exp_{G(\C_p)}\big(\log_{G(C_p)}(\gamma^k)\big)=\gamma^k\in G(\Q).$$
Using Lemma \ref{lm}, there exists a commutative algebraic subgroup $G_u$ of $G$ defined over $\Q$ of positive dimension with $\exp_{G(\C_p)}(u)\in G_u(\Q)$. This also gives $\exp_{G(\C_p)}(u)\in G_u(\C_p)_f$, and hence 
$$\log_{G(\C_p)}(\gamma^k)=u\in \Lie(G_u(\C_p))\subseteq \Lie(G(\C_p)).$$
On the other hand, it follows from \cite[Chapter III, 7.6]{bo} that
$$\log_{G(\C_p)}(\gamma^k)=k\log_{G(\C_p)}(\gamma)\in V_\gamma\otimes_{\Q}\C_p.$$ 
In particular, this implies that $V_{\gamma}\subseteq \Lie(G_u)$. This enables us to apply the $p$-adic analytic subgroup theorem (\cite[Theorem 2.2]{fp1}) to the commutative algebraic group $G_u$, the algebraic point $\gamma^k$ and the $\Q$-vector space $V_\gamma$ to obtain an algebraic (commutative) subgroup $H$ of $G$ defined over $\Q$ of positive dimension such that $\Lie(H)\subseteq V_{\gamma}$ and $\gamma^k\in H(\Q)$. Let $H^0$ be the connected component of $H$, then $\Lie(H^0)=\Lie(H)$ and there is a positive integer $m$ for which $(\gamma^k)^m\in H^0(\Q)$. Put $n=km$ and take the algebraic subgroup $H_{\gamma}$ of $G_u$ with $H_{\gamma}(\Q)=\{\alpha\in G_u(\Q): \alpha^n\in H^0(\Q)\}$. Then $\gamma\in H_\gamma$ and $\Lie(H_\gamma)\subseteq V_\gamma$. This leads to 
$$\log_{G(\C_p)}(\gamma)=\log_{H_\gamma(\C_p)}(\gamma)\in \Lie(H_\gamma(\C_p))=\Lie(H_\gamma)\otimes_{\Q}\C_p.$$ 
By the minimality of the $\Q$-vector space $V_\gamma$, we conclude that $\Lie(H_\gamma)=V_\gamma$. The theorem is therefore proved.
\end{proof}
\noindent {\bf Remark.} By the same method given in the proof of \cite[Theorem 1.5]{y} (which is deduced as a consequence of the non-commutative analytic subgroup theorem), we just use the $p$-adic exponential map instead of the complex one to also obtain the following result as a consequence of the $p$-adic non-commutative analytic subgroup theorem.
\begin{theorem}
Let $G$ be an algebraic group defined over $\Q$ and $\exp_{G(\C_p)}: U_p\rightarrow G(\C_p)$ be the $p$-adic exponential map of $G$. Let $u$ be a non-zero element of $U_p\cap \Lie(G)$ such that $\exp_{G(\C_p)}(u)$ lies in $G(\Q)$.
\\
{\rm (1)} The element $\exp_{G(\C_p)}(u)$ is contained in a subgroup isomorphic to the additive group $\mathbb G_a$.
\\
{\rm (2)} Assume $G$ is an affine algebraic group. The element $u$ is a nilpotent element of $\Lie(G(\C_p))$, and hence $\exp_{G(\C_p)}(u)$ is a unipotent element of $G(\C_p)$.
\end{theorem}

\vspace*{0.5cm}

\noindent{\sc Duc Hiep Pham}\\
University of Education\\Vietnam National University, Hanoi\\ 144 Xuan Thuy, Cau Giay, Hanoi\\ Vietnam\\
Email: {\sf phamduchiep@vnu.edu.vn}


\begin{thebibliography}{99}

\bibitem{bw} A. Baker and G. W\"{u}stholz, \emph{Logarithmic Forms and Diophantine Geometry}, New Mathematical Monographs, Cambridge University Press, 2007.

\bibitem{bo} N. Bourbaki, \emph{Elements of Mathematics. Lie groups and Lie algebras. Part I:  Chapters 1-3. English translantion.,} Actualities scientifiques et industrielles, Herman. Adiwes International Series in Mathematics. Paris: Hermann, Publishers in Arts and Science; Reading, Mass.: Addison-Wesley Publishing Company. XVII, 1975.

\bibitem{fp1} C. Fuchs and D. H. Pham, \emph{The $p$-adic analytic subgroup theorem revisited}, $p$-Adic Number, Ultrametric Analysis and Applications {\bf 7} (2015), 143--156.

\bibitem{m} J. Milne, \emph{Algebraic Groups}, Cambride Studies in Advanced Mathematics, Cambride University Press, 2017.

\bibitem{sch} P. Schneider, \emph{$p$-adic Lie groups}, Grundlehren der mathematischen Wissenschaften, Springer Verlag, 2011.

\bibitem {pt} P. Tretkoff, \emph{Periods and Special Functions in Transcendence}, Advanced Textbooks in Mathematics, World Scientific, 2017.

\bibitem{w1} G. W\"{u}stholz, \emph{Some remarks on a conjecture of Waldschmidt}, Diophantine approximations and transcendental numbers, Birkh\"{a}user  Boston, Boston, MA 1983, 329--336.

\bibitem{w3} G. W\"{u}stholz, \emph{Multiplicity estimates on group varieties}, Ann. of Math. {\bf 129} (1989), 471--500.

\bibitem{w4} G. W\"{u}stholz, \emph{Algebraische Punkte auf Analytischen Untergruppen algebraischer Gruppen}, Ann. of Math. {\bf 129} (1989), 501--517.

\bibitem{w5} G. W\"{u}stholz, \emph{Elliptic and abelian period spaces}, Acta Arith. {\bf 198} (2021), 329--357.

\bibitem{y} A. Yafaev, \emph{Non-commutative analytic subgroup theorem}, J. Number Theory (2021), doi: https://doi.org/10.1016/j.jnt.2021.01.025.


\end{thebibliography}
\end{document}